\documentclass{amsart}
\usepackage{setspace, amsmath, amsthm, amssymb, amsfonts, amscd, epic, graphicx, ulem, dsfont}
\usepackage[T1]{fontenc}
\usepackage{multirow}
\usepackage{bbm}
\usepackage{enumerate}

\makeatletter \@namedef{subjclassname@2010}{
  \textup{2010} Mathematics Subject Classification}
\makeatother

\newtheorem{thm}{Theorem}[section]

\newtheorem{cor}[thm]{Corollary}
\newtheorem{lem}[thm]{Lemma}
\newtheorem{pro}[thm]{Proposition}

\theoremstyle{remark}
\newtheorem*{rema}{Remark}
\newtheorem*{remas}{Remarks}

\theoremstyle{definition}
\newtheorem*{defn}{Definition}

\newtheorem{exas}[thm]{\textbf{Examples}}
\newcommand{\ran}{\text{\rm{ran}}}

\newcommand{\R}{\mathbb{R}}

\newcommand{\N}{\mathbb{N}}

\newcommand{\C}{\mathbb{C}}

\begin{document}

\title{On the Operator Equations $A^n=A^*A$}
\author[S. DEHIMI et al.]{Souheyb Dehimi, Mohammed Hichem Mortad$^*$ and Zsigmond Tarcsay}

\thanks{* Corresponding author.}
\date{}

\dedicatory{}


\address{(The first author): University of Mohamed El Bachir El Ibrahimi, Bordj Bou Arreridj.
Algeria.}

\email{sohayb20091@gmail.com}

\address{(The corresponding author) Department of
Mathematics, University of Oran 1, Ahmed Ben Bella, B.P. 1524, El
Menouar, Oran 31000, Algeria.}

\email{mhmortad@gmail.com, mortad.hichem@univ-oran1.dz.}

\address{(The third author) Department of Applied Analysis, E\"{o}tv\"{o}s L. University,
P\`{a}zm\`{a}ny Péter Sét\`{a}ny 1/c., Budapest H-1117,
Hungary.}

\email{tarcsay@cs.elte.hu}

\begin{abstract}
Let $n\in\N$ and let $A$ be a closed linear operator (everywhere
bounded or unbounded). In this paper, we study (among others)
equations of the type $A^*A=A^n$ where $n\geq2$ and see when they
yield $A=A^*$ (or a weaker class of operators). In case $n\geq3$, we
have in fact a new class of operators which could placed right after
orthogonal projections and just before normal operators.
\end{abstract}

\maketitle

\section{Introduction}

It was asked in \cite{Laberteux-A*A=A2} whether $A^*A=A^2$ entails
the self-adjointness of $A\in B(H)$? This was first answered
affirmatively in \cite{Wang-Zhang} on finite dimensional vector
spaces. Then, the authors in \cite{McCullough-Rodman-A*A=A2} (who
were probably not aware of \cite{Wang-Zhang}) too obtained positive
results in both the finite and the infinite dimensional settings. It
is noteworthy that the infinite dimensional case was only alluded
superficially in \cite{Wang-Zhang} where the authors of that paper
were informed by the referee of the possibility of obtaining the
self-adjointness of $A$ by using the so-called "technique of
sequences of local inverses". In this paper, we carry on this
interesting investigation to deal with the unbounded case and we
reprove some known results using simpler arguments. Some
consequences are also given. Then, we treat the more general
equations of the type
\[A^*A=A^n,~n\in\N,~n\geq3.\]

Finally, we assume readers are familiar with notions and results in
operator theory. Some general references are \cite{FUR.book},
\cite{Mortad-Oper-TH-BOOK-WSPC}, \cite{RUD} and
\cite{SCHMUDG-book-2012}.

\section{\textbf{The equations $A^*A=A^n$ with $A\in B(H)$}:}

\begin{defn}
Let $A\in B(H)$. If $A^*A=A^n$ for some $n\in\N$ such that $n\geq3$,
then $A$ is called a \textbf{generalized projection}.
\end{defn}

\begin{remas}\hfill
\begin{enumerate}
  \item First, we note that for a general $n\in\N$ (with $n\geq3$), then
$A^*A=A^n$ does not always gives the self-adjointness of $A\in B(H)$
even when $\dim H<\infty$ as we shall shortly see. Second, we notice
that if $A$ is any orthogonal projection, then it does satisfy
$A^*A=A^n$ for any $n\geq2$.
  \item In general, there are unitary operators which do not satisfy such equations even when $\dim H<\infty$. We need to find a unitary $A\in B(H)$ such
that $A^*A\neq A^n$ for any $n$. Consider on a finite dimensional
space $H$, the following:
\[A=e^{ie\pi}I\]
where $e$ is the usual transcendental number. Then $A^*A=I$ whilst
\[A^n=e^{ine\pi}I\neq I=A^*A,~\forall n\in\N,~n\geq 2.\]
\end{enumerate}
\end{remas}

The first major result of the paper is a complete characterization
of this apparently new class of operators.

\begin{thm}\label{MAIn THM}
Let $H$ be a  complex Hilbert space and let $A\in B(H)$ be a bounded
operator and let $n\in\N, n\geq2$. Then $A$ is a solution of the
equality
\begin{equation}\label{E:main}
A^n=A^*A
\end{equation}
if and only if
\begin{itemize}
    \item $A=A^*$ (if $n=2$),
    \item there is a family $P_1,\ldots,P_n\in B(H)$ of orthogonal projections such that $P_jP_k=0, (j\neq k)$ such that
    \begin{equation}\label{PPPPP}
        A=\sum_{k=1}^n e^{\frac{2k\pi i}{n}} P_k
    \end{equation}
    (if $n\geq 3$). In this case, we also have $\|A\|=1$ (when $A\neq0$).
\end{itemize}
\end{thm}
\begin{proof} The ''if'' part of the statement is clear. We  show that the ``only if'' part is also true.
First we are going to prove that
\begin{equation}\label{E:1}
    A^{n-1}=A^*.
\end{equation}
It is clear from the hypothesis that $A^*=A^{n-1}$ on the range of
$A$, hence $A^*=A^{n-1}$ also on $\overline{\ran A}$ because of
continuity. It suffices therefore to prove that $A^{n-1}=0$ on $\ker
A^{*}$, i.e., $\ker A^*\subseteq \ker A^{n-1}$. First we claim that
\begin{equation}\label{E:ran A3}
    \overline{\ran A^n}=\overline{\ran A}.
\end{equation}
Indeed, our assumption clearly implies that $\ran A^{n+1}=\ran
AA^*A=A(\ran A^*A)$, hence by equality $\overline{\ran A^*A}=\ker
A^{\perp}$ we conclude that
\begin{equation*}
    \overline{\ran A^{n+1}}=\overline{A(\ker A^{\perp})}=\overline{\ran A},
\end{equation*}
that clearly gives \eqref{E:ran A3}. From this we conclude that
\begin{equation*}
    \ker A^*=(\ran A)^{\perp}=(\ran A^n)^\perp=(\ran A^*A)^\perp= \ker A\subseteq \ker A^{n-1},
\end{equation*}
which proves \eqref{E:1}.

If $n=2$ then \eqref{E:1} expresses just that $A$ is self-adjoint.
(Observe that up to this point we did not used that $H$ is complex).
Suppose now that $n\geq 3$, then from \eqref{E:1} it follows that
$A$ is normal and that the function
\begin{equation*}
    \varphi(z)=z^{n-1}-\bar z,~z\in\C
\end{equation*}
vanishes on $\sigma(A)$. In particular, if $\lambda\in\sigma(A)$
then either $\lambda=0$ or $\lambda$ is a solution of $\lambda^n=1$,
whence we conclude that
\begin{equation*}
    \sigma(A)\subseteq \{0\}\cup\{e^{\frac{2k\pi  i}{n}},~k=1,\ldots,n\}.
\end{equation*}
Let us denote by $E$ the spectral measure of $A$ and set $P_k:=
E(\{e^{\frac{2k\pi  i}{n}}\})$ then it follows from the spectral
theorem that $P_kP_j=0$ $(k\neq j)$ and that
\begin{equation*}
    A=\sum_{k=1}^n e^{\frac{2k\pi  i}{n}} P_k.
\end{equation*}

To show the last claim, just apply the spectral radius theorem to
the normal operator $A$ to obtain $\|A\|=1$ when $A\neq0$. This
marks the end of the proof.
\end{proof}

\begin{rema}As alluded to above, any orthogonal projection satisfies the equations
$A^*A=A^n$ with $n\geq3$ ($n=2$ is also allowed). This new class of
operators lies therefore just between orthogonal projections and
normal operators.
\end{rema}

\begin{cor}\label{mpmp}
Let $A\in B(H)$ be satisfying $AA^*=A^2$. Then $A$ is self-adjoint.
\end{cor}

\begin{rema}From Theorem \ref{MAIn THM}, it turns out that operators
satisfying $A^2=A^*A$ are just the self-adjoint ones. However, a
solution of $A^n=A^*A$, $n\geq3$ need not be self-adjoint, as it can
be seen immediately from the general form (2) of those operators.
\end{rema}

As an immediate consequence, we have:

\begin{pro}
If $B,C\in B(H)$ are such that $C^*C=BC$ and $B^*B=CB$, then
$B=C^*$.
\end{pro}

\begin{proof}Let $A\in B(H\oplus H)$ be defined as $A=\left(
                                                        \begin{array}{cc}
                                                          0 & B \\
                                                          C & 0 \\
                                                        \end{array}
                                                      \right)
$. Then
\[A^*A=\left(
                                                        \begin{array}{cc}
                                                          C^*C & 0 \\
                                                          0 & B^*B \\
                                                        \end{array}
                                                      \right)\text{ and } A^2=\left(
                                                        \begin{array}{cc}
                                                          BC & 0 \\
                                                          0 & CB \\
                                                        \end{array}
                                                      \right).\]
By hypothesis, we ought to have $A^*A=A^2$, whereby $A$ becomes
self-adjoint, in which case, $B=C^*$, as wished.
\end{proof}

\begin{cor}Let $B,C\in B(H)$ be self-adjoint and such that $C^2=BC$ and $B^2=CB$. Then
$B=C$.
\end{cor}

Another consequence is the following:

\begin{pro}
Let $A\in B(H)$ be such that $A^*A^2=A^*AA^*$. Then $A$ is
self-adjoint.
\end{pro}

\begin{proof}We may write
\[A^*A^2=A^*AA^*\Longrightarrow A^*A(A-A^*)=0.\]
Hence for any $x\in H$, we clearly have that $(A-A^*)x\in\ker
(A^*A)$. But it is well known that $\ker (A^*A)$ coincides with
$\ker A$. Therefore, $(A-A^*)x\in\ker A$ or simply
\[A^2=AA^*.\]
A glance at Corollary \ref{mpmp} finally gives the self-adjointness
of $A$, marking the end of the proof.
\end{proof}

The method of matrices of operators allows us to establish the
following result:

\begin{pro}
Let $A\in B(H)$ be satisfying
\[A^*A={A^*}^2A^2=A^3.\]
Then there exist three orthogonal projections, $P_0$, $P_1$, $P_2\in
B(H)$ which are pairwise orthogonal such that
\begin{equation}\label{EQ}
A=P_0+e^{\frac{2\pi i}{3}}P_1+e^{\frac{4\pi i}{3}}P_2.
\end{equation}
\end{pro}

\begin{proof}Let $A\in B(H)$ and define $B\in B(H\oplus H)$ by:
\[B=\left(
      \begin{array}{cc}
        0 & A \\
        A^2 & 0 \\
      \end{array}
    \right).
\]
Then $B^2=\left(
            \begin{array}{cc}
              A^3 & 0 \\
              0 & A^3 \\
            \end{array}
          \right)$. Since $B^*B=\left(
                                  \begin{array}{cc}
                                    {A^*}^2A^2 & 0 \\
                                    0 & A^*A \\
                                  \end{array}
                                \right)$, by hypothesis we must therefore
                                have $B^*B=B^2$. Hence, $B$ is
                                self-adjoint by
                                Theorem \ref{MAIn THM}. This just means that
                                $A={A^*}^2$. Consequently, $A$ is
                                obviously normal and
                                \[\varphi(z)=z-\overline{z}^2,~ z\in\C\]
vanishes on $\sigma(A)$. From that it is readily seen that if
$\lambda\in\sigma(A)$ then either $\lambda=0$ or $\lambda$ is a
solution of $\lambda^3=1$. Whence, we conclude that
\[\sigma(A)\subseteq \{0\}\cup\{e^{\frac{2k\pi  i}{3}},~k=0,1,2\}.\]

From the spectral theorem it follows that $A$ can be written as
(\ref{EQ}) for some orthogonal projections $P_0$, $P_1$, $P_2$ with
pairwise orthogonal ranges. The proof is complete.
\end{proof}

We can also treat the "skew-adjointness" case. First, we give a
result which might already be known to some readers and so it is
preferable to include a proof. Recall that a bounded hyponormal
operator having a real spectrum is self-adjoint (see e.g.
\cite{Stampfli hyponormal 1965}).

\begin{lem}\label{hyponormal purely imag sp is skew-adj LEM}
Let $A\in B(H)$ be hyponormal and having purely imaginary spectrum.
Then, $A$ is skew-adjoint (that is, $A^*=-A$).
\end{lem}

\begin{proof}Set $B=iA$ and so $B$ too is hyponormal. Hence
$\sigma(B)\subset \R$ for by assumption $\sigma(A)\subset i\R$.
Hence $B$ is self-adjoint, i.e.
\[-iA^*=B^*=B=iA,\]
i.e. $A$ is clearly skew-adjoint.
\end{proof}

Mutatis mutandis, the following result is then easily obtained:

\begin{pro}\label{A*A=-A2 BOUNDED S.A. PRO}
Let $A\in B(H)$ be satisfying $A^*A=-A^2$. Then $A$ is skew-adjoint.
\end{pro}

The following sharp result is also of interest.

\begin{pro}\label{opop}
If $A\in B(H)$ is such that $A\neq 0$ and $A^*A=qA^2$ where
$q\in\R^*$, then either $q=1$ or $q=-1$.
\end{pro}

\begin{proof}By considering the cases $q<0$ and $q>0$ separately, we may as
above establish the skew-adjointness of $A$ and self-adjointness of
$A$ respectively. Now, in case $A$ is skew-adjoint (when $q<0$), we
may write
\[A^*A=qA^2\Longrightarrow -A^2=qA^2\Longrightarrow (1+q)A^2=0\]
which gives $q=-1$. In the event of the self-adjointness of $A$, we
may just reason similarly to get $q=1$, and this finishes the proof.
\end{proof}

\section{\textbf{The equations $A^*A=A^n$ with a closed and densely defined operator $A$:}}

First, we stop by some examples.

\begin{exas}\hfill
\begin{enumerate}
  \item If $A$ is a linear operator, then $A^*A=A^2$ does not necessarily give
  $A=A^*$.   The most trivial example is to consider a densely defined and unclosed operator $A$ (hence such $A$ cannot be self-adjoint) such that
\[D(A^2)=D(A^*A)=\{0\}\]
as in \cite{Mortad-triviality domains powers adjoints}, say. Then
$A^*A=A^2$ is trivially satisfied.

  \item \textit{$A^*A\subset A^2 \not\Rightarrow A=A^*$ even when $A$ is closed}: Indeed, consider any closed, densely defined and symmetric operator $A$ \textit{which is not self-adjoint}. Then $A\subset A^*$ and so
   $A^*A\subset A^2$.
  \item \textit{$A^2\subset A^*A \not\Rightarrow A=A^*$ even when $A$ is closed}: In this case, consider any closed, densely defined and symmetric operator $A^*$ \textit{which is not self-adjoint}. A similar observation as just above then yields $A^2\subset A^*A$. We may even consider a
closed, symmetric and semi-bounded such that $D(A^2)=\{0\}$ (see
\cite{CH}, cf. \cite{Dehimi-Mortad-CHERNOFF}). Then trivially
$A^2\subset A^*A$ and $A$ is not self-adjoint.
\end{enumerate}
\end{exas}

Now, we deal with the equation $A^*A=A^2$ for a closed and densely
defined $A$.

\begin{thm}\label{A*A=A2 S.A. UNBOUNDED THM}
Let $H$ be a complex Hilbert space and let $A$ be a closed and
densely defined (unbounded) operator verifying $A^*A=A^2$. Then $A$
is self-adjoint on its domain $D(A)\subset H$.
\end{thm}

\begin{proof}Plainly,
\[A^*A=A^2\Longrightarrow AA^*A=A^{3}\Longrightarrow AA^*A=A^2A\Longrightarrow AA^*A=A^*AA,\]
showing the quasinormality of $A$ (as defined in \cite{Jablonski et
al 2014}, say). By consulting \cite{Janas-HYPO-I} and \cite{Madjak},
we know that quasinormal operators are hyponormal. That is, $A$ is
hyponormal.

According to the proof of Theorem 8 in \cite{Dehimi-Mortad-BKMS},
closed hyponormal operators having a real spectrum are automatically
self-adjoint. Once that's known and in order that $A$ be
self-adjoint, it suffices therefore to show the realness of its
spectrum given that $A$ is already closed.

So, let $\lambda\in\sigma(A)$. Since $A$ is closed, we have by
invoking a spectral mapping theorem (e.g. Theorem 2.15 in
\cite{Kulkrani et al-2008}) that $\lambda^2\geq0$ for $A^*A$ is
self-adjoint and positive. Now,  this forces $\lambda$ to be real.
Accordingly, $\sigma(A)\subset \R$, as needed.
\end{proof}

\begin{rema}
Notice that the previous proof may well be applied to the first
claim of Theorem \ref{MAIn THM} when $H$ is a Hilbert space over
$\C$.
\end{rema}

As in the bounded case, we have:

\begin{pro}
Let $B,C$ be two densely defined and closed operators obeying
$C^*C=BC$ and $B^*B=CB$. Then $B=C^*$.
\end{pro}

By adopting a very similar idea to the bounded case (by observing
that Lemma \ref{hyponormal purely imag sp is skew-adj LEM} holds for
unbounded and closed operators as well), we may easily establish the
following result. We include, however, a somewhat different proof
which could have been used above anyway.

\begin{pro}
Let $A$ be a closed and densely defined (unbounded) operator such
that $A^*A=-A^2$. Then $A$ is skew-adjoint.
\end{pro}

\begin{proof}
Set $B=iA$. Then
\[A^*A=-A^2\Longrightarrow B^*B=B^2.\]
Since $B$ is closed, Theorem \ref{A*A=A2 S.A. UNBOUNDED THM} applies
and gives the self-adjointness of $B$ or the skew-adjointness of
$A$, as required.
\end{proof}

An unbounded version of Proposition \ref{opop} is also available.

\begin{pro}
If $A$ is a closed, unbounded and densely defined operator such that
 $A^*A=qA^2$ where $q\in\R^*$, then either $q=1$ or
$q=-1$.
\end{pro}

\begin{proof}The proof is similar to the one in the bounded case.
For instance, when $q>0$, we obtain the self-adjointness of $A$.
Hence
\[A^*A=qA^2\Longrightarrow A^2=qA^2\Longrightarrow (1-q)A^2\subset 0\]
which forces $q=1$ (remember that $A^2$ is unbounded).
\end{proof}

Finally, we treat the unbounded case. Somehow expectedly, we show
the impossibility of the equations $A^*A=A^n$ (with $n\geq3$) for
unbounded closed operators.

\begin{thm}\label{A*A=An n geq 3UNBOUNDED}
Let $A$ be a closed and densely defined operator with a domain
$D(A)\subset H$ and let $n\in\N$ be such that $n\geq3$. If
$A^*A=A^n$, then $A\in B(H)$ (and so $A$ can be written in the form
(\ref{PPPPP})).
\end{thm}

\begin{proof}Let $A$ be a closed and densely defined operator which
obeys $A^*A=A^n$ where $n\geq3$. Then (as in the bounded case)
\[A^*A=A^n\Longrightarrow AA^*A=A^{n+1}\Longrightarrow AA^*A=A^nA\Longrightarrow AA^*A=A^*AA,\]
showing the quasinormality of $A$. It then follows that $A$ is
hyponormal and so $D(A)\subset D(A^*)$. Hence
\[
D\left( A^{2}\right) \subseteq D\left( A^{\ast }A\right)=D\left(
A^{n}\right)\]
or merely
\[D\left( A^{2}\right) =D\left( A^{n}\right).\]
Also
\[D\left( A^{3}\right) \subseteq D\left( A^{\ast }AA\right) =D\left(
A^{n+1}\right)\] so that
\[D\left( A^{2}\right)=D\left(A^{n+1}\right).\]
Now, since $A$ is closed, it follows that $A^2$ is closed as it is
already quasinormal (see e.g. Proposition 5.2 in \cite{Stochel-P(S)
closed IEOT-2002}). Also, the quasinormality of $A$ yields that of
$A^{2}$ (by Corollary 3.8 in \cite{Jablonski et al 2014}, say) and
so $A^2$ is hyponormal. Therefore,
\[
D\left(A^{2}\right) \subseteq D[\left(A^{2}\right)^*]\] and
\[
D\left(A^{2}\right) =D\left(A^{4}\right).\]

In the end, according to Corollary 2.2 in \cite{Tarcsay-2012-bounded
D(T)=D(T2)}, it follows that $A^{2}$ is everywhere bounded on $H$.
Hence $D(A)=H$ and so the Closed Graph Theorem intervenes now to
make $A\in B(H)$, as coveted.
\end{proof}

\end{document}